\documentclass[12pt,twoside]{amsart}
\usepackage{latexsym,amsfonts,amsmath,amssymb,soul}
\usepackage{enumerate,soul}
\usepackage[usenames, dvipsnames]{color}
\numberwithin{equation}{section}

\makeatletter
\@namedef{subjclassname@2020}{%
  \textup{2020} Mathematics Subject Classification}
\makeatother

\newtheorem{Th}{Theorem}[section]
\newtheorem{Rem}[Th]{Remark}
\newtheorem{Ex}[Th]{Example}
\newtheorem{Lemma}[Th]{Lemma}
\newtheorem{Def}[Th]{Definition}
\newtheorem{Prop}[Th]{Proposition}
\newtheorem{Cor}[Th]{Corollary}

\catcode`\@=11

\renewcommand{\section}%
   {\setcounter{equation}{0}\@startsection {section}{1}{\z@}{-3.5ex plus -1ex
  minus -.2ex}{2.3ex plus .2ex}{\Large\bf}}

\def\ds{\displaystyle}
\def\R{\mathbb R}
\def\C{\mathbb C}
\def\N{\mathbb N}

\newcommand{\Lin}{\mathcal{L}}
\newcommand{\Sch}{\mathcal{S}}
\newcommand{\calM}{\mathcal{M}}
\newcommand{\calN}{\mathcal{N}}
\newcommand{\bfM}{\mathbf{M}}
\newcommand{\bfN}{\mathbf{N}}
\newcommand{\bfc}{\mathbf{c}}

\def\lc{\mathop{\rm lc}\nolimits}

\newcommand{\beqsn}{\arraycolsep1.5pt\begin{eqnarray*}}
\newcommand{\eeqsn}{\end{eqnarray*}\arraycolsep5pt}
\newcommand{\beqs}{\arraycolsep1.5pt\begin{eqnarray}}
\newcommand{\eeqs}{\end{eqnarray}\arraycolsep5pt}

\title{Construction of the log-convex minorant of a sequence 
$\{M_\alpha\}_{\alpha\in\mathbb N_0^d}$}

\author[Boiti]{Chiara Boiti}
\address{
Dipartimento di Matematica e Informatica \\Universit\`a di Ferrara\\
Via Ma\-chia\-vel\-li n.~30\\
I-44121 Ferrara\\
Italy}
\email{chiara.boiti@unife.it}

\author[Jornet]{David Jornet}
\address{
Instituto Universitario de Matem\'atica Pura y Aplicada IUMPA\\
Universitat Po\-li\-t\`ecni\-ca de Val\`encia\\
Camino de Vera, s/n\\
E-46022 Valencia\\
Spain}
\email{djornet@mat.upv.es}

\author[Oliaro]{Alessandro Oliaro}
\address{Dipartimento di Matematica\\ Universit\`a di Torino\\
 Via Carlo Alberto n.~10\\ I-10123 Torino\\ Italy}
 \email{alessandro.oliaro@unito.it}

\author[Schindl]{Gerhard Schindl}
\address{Fakult\"at f\"ur Mathematik\\ Universit\"at Wien\\
Oskar-Morgenstern-Platz n.~1\\ A-1090 Wien\\ Austria}
 \email{gerhard.schindl@univie.ac.at}

\begin{document}

\keywords{regularization of sequences, log-convex sequences, ultradifferentiable functions, matrix weights}
\subjclass[2020]{40B05, 46A13, 46A45}

\begin{abstract}
We give a simple construction of the log-convex minorant of a sequence
$\{M_\alpha\}_{\alpha\in\mathbb N_0^d}$ and consequently extend to the 
$d$-dimensional case the
well-known formula that relates a log-convex sequence $\{M_p\}_{p\in\mathbb N_0}$ to its associated function $\omega_M$, that is $M_p=\sup_{t>0}t^p\exp(-\omega_M(t))$. We show that
in the more dimensional anisotropic case the
classical log-convex condition $M_\alpha^2\leq M_{\alpha-e_j}M_{\alpha+e_j}$ is not sufficient: convexity as a function of more variables is needed (not only coordinate-wise).
We finally obtain some applications to the inclusion of spaces of rapidly decreasing ultradifferentiable functions in the matrix weighted setting.
\end{abstract}

\maketitle

%

\markboth{\sc  Construction of the log-convex minorant\ldots}
 {\sc C.~Boiti, D.~Jornet, A.~Oliaro, G.~Schindl}

\section{Introduction}
For a sequence $\{M_p\}_{p\in\N_0}$ of real positive numbers (with $M_0=1$ for simplicity; 
$\N_0:=\N\cup\{0\}$), its 
{\em associated function} is defined by
\beqsn
\omega_M(t):=\sup_{p\in\N_0}\log\frac{t^p}{M_p},\qquad t>0.
\eeqsn
Mandelbrojt proved in \cite[Chap.~I]{M}  (see also \cite{K}) that if 
$\lim_{p\to+\infty}M_p^{1/p}=+\infty$ then
\beqs
\label{In1}
M_p=\sup_{t>0}\frac{t^p}{\exp\omega_M(t)},\qquad p\in\N_0,
\eeqs
if and only if $\{M_p\}_{p\in\N_0}$ is logarithmically convex, i.e.
\beqsn
M_p^2\leq M_{p-1}M_{p+1},\qquad\forall p\in\N.
\eeqsn
We refer also to the recent work \cite{S} where this construction has been studied again in detail, some technical ambiguities have been solved and non-standard cases have been studied as well; see Section \ref{sec2}.

However, to the best of our knowledge this condition has never been generalized to the $d$-dimensional anisotropic case ($d>1$), and the reason is that the classical coordinate-wise logarithmic convexity condition \eqref{LOG-C} for
a sequence $\{M_\alpha\}_{\alpha\in\N^d_0}$ 
is not sufficient to obtain the analogous of \eqref{In1} for $M_\alpha$, as explained in Remark~\ref{rem4}. The reason is that this is a convexity
condition on each variable separately and not on the globality of its variables.
Assuming the stronger condition that $\{M_\alpha\}_{\alpha\in\N^d_0}$ is log-convex on the globality of its variables $\alpha\in\N_0^d$ (see Definition~\ref{defcvx}),
we extend \eqref{In1} to $\alpha\in\N_0^d$ instead of $p\in\N_0$ (see
Theorem~\ref{cor1}).

To obtain this result we construct in Sections~\ref{sec1}-\ref{sec3} the (optimal) convex minorant of a
sequence $\{a_\alpha\}_{\alpha\in\N^d_0}$ (then $a_\alpha=\log M_\alpha$ in Section~\ref{sec4}).
The idea, in the one variable case, takes inspiration from the convex regularization of sequences of Mandelbrojt in \cite{M}, which was quite complicated and difficult to export to the more
dimensional case. Our construction is made by taking the supremum of hyperplanes approaching from below the given sequence and leads to the notion of convexity for a sequence
$\{a_\alpha\}_{\alpha\in\N^d_0}$ in the sense that $a_\alpha=F(\alpha)$ for a convex function
$F:[0,+\infty)^d\to\R$. This condition gives the suitable notion of logarithmic convexity for
a sequence $\{M_\alpha\}_{\alpha\in\N^d_0}$ in order to write it in terms of its associated function
as in \eqref{Q3}.

This result is a very useful tool for working in the anisotropic setting.
In Section~\ref{sec5}, indeed, we obtain some applications about inclusion of spaces of rapidly decreasing ultradifferentiable functions in the matrix weighted anisotropic setting, where
Theorem~\ref{cor1} is crucial (see Remark~\ref{rem62} and compare with the isotropic case in \cite{BJOS-fuzzy}).
In particular, we characterize the conditions on the matrix weights $\calM$ and $\calN$ in order to have a continuous inclusion between the spaces
$\Sch_{\{\calM\}}/\Sch_{(\calM)}$, $\Sch_{\{\calN\}}/\Sch_{(\calN)}$, both in the Roumieu and Beurling cases.
The advantages of the matrix weighted setting was already enlightened  in \cite{RS},  in order to treat at the same time classes in the sense of Komatsu \cite{K} (estimates of the derivatives with a sequence) and in the sense of Braun, Meise and Taylor  \cite{BMT} (estimates of the derivatives via a weight function). 
Since then several papers using weight matrices have been published. We mention, for instance, \cite{BJOS-Banach,almostanalytic,furdos2022the,mixedramisurj}, and the references therein.

We refer 
to \cite{BMM} to compare classes of ultradifferentiable functions as defined by
Komatsu \cite{K} and as defined by Braun, Meise and Taylor \cite{BMT}, and
we point out the recent papers \cite{W, RW1, RW2, ACT, DC} for other different interesting
results in anisotropic Gelfand-Shilov spaces.

\section{Construction of the convex minorant candidate}
\label{sec1}

Let us recall that a real sequence $\{a_p\}_{p\in\N_0}$ is said to be {\em convex} if
\beqsn
a_p\leq\frac12 a_{p-1}+\frac12a_{p+1},\qquad\forall p\in\N.
\eeqsn
This is equivalent to say that the polygonal obtained by connecting the points
$(p,a_p)$ with $(p+1,a_{p+1})$ by a straight line, for all $p\in\N_0$, is the graph of a convex function, or equivalently that there exists a convex function $F:\,[0,+\infty)\to \R$ with
$F(p)=a_p$ for all $p\in\N_0$.

This suggests for
a real sequence $\{a_\alpha\}_{\alpha\in\N_0^d}$ the following:
\begin{Def}
\label{def00}
We say that a sequence $\{a_\alpha\}_{\alpha\in\N_0^d}$ is {\em convex} if there exists a
convex function $F:\,[0,+\infty)^d\to\R$ with $F(\alpha)=a_\alpha$ for all $\alpha\in\N_0^d$.
\end{Def}
\begin{Def}
\label{def01}
The {\em convex minorant} of a sequence
 $\{a_\alpha\}_{\alpha\in\N_0^d}$ is the largest convex sequence
 $\{a^c_\alpha\}_{\alpha\in\N_0^d}$ with $a^c_\alpha\leq a_\alpha$ for all $\alpha\in\N_0^d$.
\end{Def}

We want to construct the convex minorant of a sequence
$\{a_\alpha\}_{\alpha\in\N_0^d}\subset\overline{\R}:=\R\cup\{\pm\infty\}$ such that:
\begin{enumerate}[$(i)$]
\item
\vspace*{1mm}
$\ds a_\alpha>-\infty,\quad\forall\alpha\in\N_0^d,$
\item
\vspace*{1.5mm}
$\ds\lim_{|\alpha|\to+\infty}\frac{a_\alpha}{|\alpha|}=+\infty,
\quad\mbox{for}\ |\alpha|=\alpha_1+\ldots+\alpha_d,$
\item
\vspace*{1mm}
$\ds a_\alpha$ may be $+\infty$ at most for a finite number of multi-indices $\ds\alpha\in\N_0^d$,
\item
\vspace*{1mm}
$\ds a_0\in\R.$
\end{enumerate}

To this aim we set
\beqsn
S:=&&\{(\alpha,a_\alpha):\ \alpha\in\N_0^d\},\\
\Lin:=&&\{f:\R^d\to\R:\ f\ \mbox{is an affine function}\}\\
=&&\{f(x)=\langle k,x\rangle+c:\ (k,c)\in\R^d\times\R\}.
\eeqsn
Note that the graphs of the affine functions $f\in\Lin$ are hyperplanes in $\R^{d+1}$.

\begin{Lemma}
\label{lemma1}
Given $f\in\Lin$ we have 
\beqsn
f(\alpha)>a_\alpha,
\eeqsn
at most for a finite number of points $\alpha\in\N_0^d$.
\end{Lemma}

\begin{proof}
For $|\alpha|\geq 1$ by the Cauchy-Schwarz inequality we have 
\beqsn
\frac{|f(\alpha)|}{|\alpha|}\leq
\frac{|\langle k,\alpha\rangle|+|c|}{|\alpha|}\leq
\frac{\|k\|\cdot\|\alpha\|}{|\alpha|}+|c|
\leq\|k\|+|c|,
\eeqsn
since $\|\alpha\|=\sqrt{\alpha_1^2+\ldots+\alpha_d^2}\leq\alpha_1+\ldots+\alpha_d=|\alpha|$.

On the other hand, assumption $(ii)$ implies 
\beqsn
\frac{a_\alpha}{|\alpha|}>\|k\|+|c|,
\eeqsn
for $|\alpha|$ large enough. This implies that only a finite number of $\alpha\in\N_0^d$ may satisfy that $|f(\alpha)|>a_\alpha$.
\end{proof}

Let us now consider
\beqsn
\Lin_S:=\{f\in\Lin:\ f(\alpha)\leq a_\alpha\ \forall\alpha\in\N_0^d\}.
\eeqsn
The graphs of the functions $f\in\Lin_S$ are the hyperplanes which lie under $S$.

Note that $\Lin_S\neq\emptyset$ by Lemma~\ref{lemma1}. As a matter of fact, given $f\in\Lin$, since $f(\alpha)\leq a_\alpha$ except a finite number of points $\alpha_1,\ldots,\alpha_\ell$, we have that
\beqsn
f-\max_{1\leq j\leq\ell}\{f(\alpha_j)-a_{\alpha_j}\}\in\Lin_S.
\eeqsn

The idea is now to consider the supremum of these hyperplanes which lie under $S$ and project on this set each $a_\alpha$ to construct the desired convex minorant sequence.
So let us first define
\beqs
\label{defF}
F(x):=\sup_{f\in\Lin_S}f(x).
\eeqs

Note that the above defined $F:\R^d\to\R$ is a convex function since it is the supremum of a set of convex (affine) functions: the epigraph is convex because it is the intersection of convex sets.
By construction
\beqsn
F(\alpha)\leq a_\alpha,\qquad\forall\alpha\in\N_0^d,
\eeqsn
and we claim that
\beqsn
F(x)<+\infty\qquad\forall x\in[0,+\infty)^d.
\eeqsn
As a matter of fact, given $x\in[0,+\infty)^d$, by assumptions $(i)$ and $(iii)$, we can find $d+1$
points $\alpha_1,\ldots,\alpha_{d+1}\in\N_0^d$ such that $x$ is inside
(or on the border of) the simplex of vertices $\alpha_1,\ldots,\alpha_{d+1}$ and
$f(\alpha_j)\leq a_{\alpha_j}<+\infty$
for every $f\in\Lin_S$ and $1\leq j\leq d+1$. Then
\beqsn
f(x)\leq\max_{1\leq j\leq d+1}f(\alpha_j)<+\infty,
\eeqsn
by the convexity of $f$, and hence $F(x)<+\infty$.

Note also that $F$ is continuous on $(0,+\infty)^d$, being a convex function.

For fixed $k\in\R^d$ let us now define
\beqs
\label{hk1}
h_k:=&&\sup\{c\in\R:\ f(x)=\langle k,x\rangle+c\in\Lin_S\}\\
\nonumber
=&&\sup\{c\in\R:\ \langle k,\alpha\rangle+c\leq a_\alpha\,,\ \forall\alpha\in\N_0^d\}.
\eeqs
We have:

\begin{Lemma}
\label{lemma2}
Let $k\in\R^d$ and $h_k$  as in \eqref{hk1}. Then
\begin{equation}
\label{hk2}
h_k=\inf_{\alpha\in\N_0^d}\{a_\alpha-\langle k,\alpha\rangle\}
=\min_{\alpha\in\N_0^d}\{a_\alpha-\langle k,\alpha\rangle\}.
\end{equation}
\end{Lemma}

\begin{proof}
Let us first remark that $\inf_{\alpha\in\N_0^d}\{a_\alpha-\langle k,\alpha\rangle\}$ is a minimum because assumption $(ii)$ implies that
\beqsn
\lim_{|\alpha|\to+\infty}(a_\alpha-\langle k,\alpha\rangle)=+\infty,
\eeqsn
(see the proof of Lemma \ref{lemma1}),
so that the infimum is attained on a bounded set of $\N_0^d$, which is a finite set.

Let us now set
\beqs
\label{hktilde}
\tilde h_k:=\inf_{\alpha\in\N_0^d}\{a_\alpha-\langle k,\alpha\rangle\}
=\min_{\alpha\in\N_0^d}\{a_\alpha-\langle k,\alpha\rangle\}=
a_{\bar{\alpha}}-\langle k,\bar{\alpha}\rangle
\eeqs
for some minimum point $\bar{\alpha}\in\N_0^d$, and prove that $\tilde h_k=h_k$ as defined in \eqref{hk1}.

Clearly
\beqsn
&&\tilde h_k\leq a_\alpha-\langle k,\alpha\rangle,\qquad\forall\alpha\in\N_0^d,\\
\Leftrightarrow\quad&&\tilde h_k+\langle k,\alpha\rangle\leq a_\alpha,
\qquad\forall\alpha\in\N_0^d.
\eeqsn
Therefore
\beqsn
f(x)=\tilde h_k+\langle k,x\rangle\in\Lin_S
\eeqsn
is one of the functions in \eqref{hk1} with $c=\tilde h_k$ and hence $h_k\geq\tilde h_k$.

On the contrary, if $c>\tilde h_k$ then \eqref{hktilde} implies that
\beqsn
c+\langle k,\bar{\alpha}\rangle>\tilde h_k+\langle k,\bar{\alpha}\rangle=a_{\bar{\alpha}}
\eeqsn
so that in this case
\beqsn
\bar{f}(x)=c+\langle k,x\rangle\notin\Lin_S
\eeqsn
and therefore $h_k\leq \tilde h_k$.
\end{proof}

Now we define
\beqs
\label{fk}
f_k(x):=h_k+\langle k,x\rangle.
\eeqs

\begin{Prop}
\label{prop1}
Let $F$ be as in \eqref{defF}, $h_k$ as in \eqref{hk2} and $f_k$ as in \eqref{fk}. Then
\beqsn
F(x)=\sup_{k\in\R^d}f_k(x).
\eeqsn
\end{Prop}

\begin{proof}
Since $f_k\in\Lin_S$ by \eqref{hk2}, we clearly have that
\beqsn
F(x)\geq\sup_{k\in\R^d}f_k(x).
\eeqsn
On the other hand, if $f\in\Lin_S$ then
\beqsn
f(x)=c+\langle k,x\rangle
\eeqsn
for some $c\in\R$, $k\in\R^d$, and hence by \eqref{hk1}
\beqsn
f(x)\leq f_k(x)\leq\sup_{\ell\in\R^d}f_{\ell}(x).
\eeqsn
This finally implies that
\beqsn
F(x)=\sup_{f\in\Lin_S}f(x)\leq\sup_{k\in\R^d}f_k(x).
\eeqsn
\end{proof}

We thus have a convex sequence $\{\tilde a_\alpha\}_{\alpha\in\N_0^d}$ defined, for
$F$ as in \eqref{defF}, by
\beqs
\label{000}
\tilde a_\alpha:=F(\alpha),
\eeqs
and which can be equivalently be defined by
\beqsn
\tilde a_\alpha:=\sup_{k\in\R^d}(\langle k,\alpha\rangle+h_k),
\eeqsn
for
\beqsn
h_k=\inf_{\alpha\in\N_0^d}(a_\alpha-\langle k,\alpha\rangle).
\eeqsn
We call it the {\em convex minorant candidate} of $\{a_\alpha\}_{\alpha\in\N_0^d}$
since we shall prove in the following sections that it is indeed the convex minorant
of $\{a_\alpha\}_{\alpha\in\N_0^d}$ (see Corollary~\ref{cor00}).

The first step in this direction is to clarify the one-dimensional case (in Section~\ref{sec2}) and
then proceed by descending induction on $d$ for the $d$-dimensional case
(in Section~\ref{sec3}). To this aim in the next section we shall look more closely at the
geometric construction in the one dimensional case, that is particularly clear.

\section{Geometric construction of the convex minorant sequence in the one-dimensional case}
\label{sec2}
The geometric construction in the one-dimensional case can also be found in \cite{S}; we refer to this work and in particular to \cite[Sect. 3.3]{S} concerning the comparison with the classical results from \cite[Chap.~I]{M}.
Here we revisit the construction in the spirit of Section~\ref{sec1} since it is required in the induction argument for the higher-dimensional situation in Section \ref{sec3}. 

Let $\{a_\alpha\}_{\alpha\in\N_0}$ satisfy $(i)-(iv)$. In particular the condition $a_0\in\R$
will be essential for the first step of the construction.
All functions of $\Lin_S$ are of the form
\beqsn
f(x)=c+kx,\qquad\mbox{with}\ k\in\R\ \mbox{and}\ c\leq a_0,
\eeqsn
since $f(\alpha)\leq a_\alpha$ for all $\alpha\in\N_0$ implies, in particular, $f(0)=c\leq a_0$.

Let us now consider the functions of $\Lin_S$ of the form
\beqsn
f_{a_0,k}(x)=a_0+kx,\qquad k\in\R,
\eeqsn
and note that
\beqsn
F(0)=\sup_{f\in\Lin_S}f(0)=a_0.
\eeqsn
The idea is now to rotate (i.e. increasing the slope $k$) this straight line $y=a_0+kx$ around the point $(0,a_0)\in S$ until we meet another point $(p,a_p)\in S$. In order to have
$f_{a_0,k}(p)=a_p$ we find
\beqsn
a_p=a_0+kp\quad\Leftrightarrow\quad k=\frac{a_p-a_0}{p}.
\eeqsn
Take then
\beqsn
k_0:=\inf_{p\in\N}\frac{a_p-a_0}{p}
\eeqsn
and note that it is a minimum because of the assumption $a_p/p\to+\infty$ for $p\to+\infty$,
which implies that the infimum can be done on a bounded subset of $\N$ and hence on a finite
number of indices $p$:
\beqsn
k_0=\min_{p\in\N}\frac{a_p-a_0}{p}=\frac{a_{p_1}-a_0}{p_1}
\eeqsn
for some $p_1\in\N$. Observe that $p_1$ does not need to be unique; if there is more than one $p_1$ realizing the minimum, for the construction below it doesn't matter which one we choose at this step.

Set
\beqs
\label{firstline}
f_{a_0,k_0}(x):=a_0+k_0x.
\eeqs
We claim that
\beqsn
F(x)=f_{a_0,k_0}(x)=a_0+k_0x,
\qquad\forall x\in[0,p_1],
\eeqsn
where $F$ is the function defined in \eqref{defF}, i.e. the geometric construction coincides with the construction made in \S\ref{sec1}, in
$[0,p_1]$.

On the one side, $f_{a_0,k_0}\in\Lin_S$ by construction and hence
\beqsn
F(x)\geq f_{a_0,k_0}(x),
\qquad\forall x\in[0,p_1].
\eeqsn
On the other side, if
\beqsn
f(x)=c+kx\in\Lin_S
\eeqsn
then we must have
\beqsn
&&f(0)\leq a_0=f_{a_0,k_0}(0)\\
&&f(p_1)\leq a_{p_1}=f_{a_0,k_0}(p_1)
\eeqsn
and hence
\beqsn
f(x)\leq f_{a_0,k_0}(x),\qquad\forall x\in[0,p_1],
\eeqsn
since they are affine functions. Therefore
\beqsn
F(x)=\sup_{f\in\Lin_S}f(x)\leq f_{a_0,k_0}(x),\qquad\forall x\in[0,p_1].
\eeqsn
We have thus proved that
\beqsn
F(x)=a_0+k_0x,\qquad\forall x\in[0,p_1],
\eeqsn
and moreover
\beqsn
F(p_1)=a_{p_1}.
\eeqsn

We can further proceed as in the previous step, considering
\beqsn
f_{a_{p_1},k}(x)=a_{p_1}+k(x-p_1)
\eeqsn
with $f_{a_{p_1},k}(p_1)=a_{p_1}=F(p_1)$.
Requiring
\beqsn
f_{a_{p_1},k}(p)=a_p\ \Leftrightarrow\ a_p=a_{p_1}+k(p-p_1)\
\Leftrightarrow\ k=\frac{a_p-a_{p_1}}{p-p_1}
\eeqsn
we take
\beqsn
k_1:=\inf_{p_1<p\in\N}\frac{a_p-a_{p_1}}{p-p_1}=
\min_{p_1<p\in\N}\frac{a_p-a_{p_1}}{p-p_1}=
\frac{a_{p_2}-a_{p_1}}{p_2-p_1},
\eeqsn
for some $p_1<p_2\in\N$, and set
\beqsn
f_{a_{p_1},k_1}(x):=a_{p_1}+k_1(x-p_1).
\eeqsn
Then, for all $p\in[p_1,p_2]$,
\beqsn
F(x)=&&f_{a_{p_1},k_1}(x)=a_{p_1}+k_1(x-p_1)\\
=&&k_1x+a_{p_1}-k_1p_1\\
=&&k_1x+a_{p_1}-\frac{a_{p_2}-a_{p_1}}{p_2-p_1}p_1\\
=&&k_1x+\frac{p_2a_{p_1}-p_1a_{p_2}}{p_2-p_1}\\
=&&k_1x+d_{k_1}
\eeqsn
with
\beqsn
d_{k_1}=\frac{p_2a_{p_1}-p_1a_{p_2}}{p_2-p_1}.
\eeqsn
Moreover $F(p_2)=a_{p_2}$. Also in this case, $p_2$ does not need to be unique, and the choice of $p_2$ does not affect the next steps.

Going on recursively in the same way we have a geometric construction which coincides with the construction of $F$ in \S\ref{sec1}.

The convex minorant candidate sequence given by $\tilde a_p=F(p)$
as defined in \eqref{000} is in this case the projection of $a_p$ on the segments of lines above defined in each interval $[p_i,p_{i+1}]$:
\beqsn
\tilde a_{p_i}=&&a_{p_i}, \qquad\forall i\in\N_0,\\
\tilde a_p=&&pk_i+d_{k_i}\\
=&&\frac{a_{p_{i+1}}-a_{p_i}}{p_{i+1}-p_i}\,p+
\frac{p_{i+1}a_{p_i}-p_ia_{p_{i+1}}}{p_{i+1}-p_i},\qquad p_i<p<p_{i+1}.
\eeqsn

In the one dimensional case it is thus immediate by the construction that the convex minorant candidate sequence $\{\tilde a_p\}_{p\in\N_0}$ coincides with the convex minorant
$\{a^c_p\}_{p\in\N_0}$.

%

\section{Convex minorant sequence in the ${d}$-dimensional case}
\label{sec3}

In order to show that the convex minorant candidate sequence $\{\tilde a_\alpha\}_{\alpha\in\N_0^d}:=
\{F(\alpha)\}_{\alpha\in\N_0^d}$ defined in \eqref{000} is the convex minorant sequence of $\{a_\alpha\}_{\alpha\in\N_0^d}$, we shall now prove that $F$ is the biggest convex function whose epigraph contains $S$.

\begin{Th}
\label{th1}
The function $F$ defined in \eqref{defF} coincides with the biggest convex function
$g:[0,+\infty)^d\to\R$ such that
\beqs
\label{galpha}
g(\alpha)\leq a_\alpha,\qquad\forall\alpha\in\N_0^d.
\eeqs
\end{Th}

\begin{proof}
Let us first remark that, since $F$ is a convex function such that
\beqsn
F(\alpha)\leq a_\alpha,\qquad\forall\alpha\in\N_0^d,
\eeqsn
the largest convex function $g$ which satisfies \eqref{galpha} must be larger than $F$:
\beqs
\label{easyineq}
g(x)\geq F(x),\qquad\forall x\in[0,+\infty)^d.
\eeqs

In order to prove the opposite inequality, let us first work in the interior of $[0,+\infty)^d$.
Fix $x^0\in(0,+\infty)^d$ and consider $(x^0,y^0)=(x^0,g(x^0))$ on the graph of $g$.

Since $g$ is convex, its epigraph $G_g^+$ is a convex subset of $\R^{d+1}$.
It follows, as a consequence of the Hahn-Banach theorem, that there is a hyperplane of the form
\beqsn
y=\langle k^*,x-x^0\rangle+y^0=\langle k^*,x\rangle+c^*,
\eeqsn
for some $k^*\in\R^d$ and $c^*=y^0-\langle k^*,x^0\rangle$, that leaves the whole set $G_g^+$
on the same side of the hyperplane. Note that $x^0\in(0,+\infty)^d$ avoids ``vertical'' hyperplanes.

Then
\beqsn
f^*(x):=\langle k^*,x\rangle+c^*\in\Lin_S,
\eeqsn
since $g(\alpha)\leq a_\alpha$ for all $\alpha\in\N_0^d$ by assumption.
It follows that
\beqsn
F(x^0)=\sup_{f\in\Lin_S}f(x^0)\geq f^*(x^0)=y^0=g(x^0),
\eeqsn
and hence, by the arbitrariness of $x^0\in(0,+\infty)^d$ and by  \eqref{easyineq}, we have that
\beqs
\label{interior}
F(x^0)=g(x^0),\qquad \forall x^0\in(0,+\infty)^d.
\eeqs

Let us now consider the case $x^0\in\partial[0,+\infty)^d$, i.e. $x^0_j=0$ for at least one
$1\leq j\leq d$. Assume for simplicity $j=d$ and set $x=(x',0)=(x_1,\ldots,x_{d-1},0)$ for
$x\in[0,+\infty)^{d-1}\times\{0\}$.

The problem in this case is that the graph of $g$ could have in $(x^0,g(x^0))$ a tangent
``vertical'' hyperplane of the form $\{x_d=0\}$ and hence not defined by a function of $\Lin_S$.
The idea is to prove that the trace of $F$ on $\{x_d=0\}$ is a supremum of affine functions on
$\R^{d-1}$ whose graph is below $S\cap\{x_d=0\}$ reducing to the case of dimension $d-1$, and then
proceed recursively up to dimension 1, where the assumption that $a_0\in\R$ guarantees the conclusion.

So let us define
\beqsn
\Lin_S^0:=\{f(x')=\langle k',x'\rangle+c:\ k'\in\R^{d-1},c\in\R,\langle k',\alpha'\rangle+c\leq a_{(\alpha',0)}\ \forall\alpha'\in\N_0^{d-1}\}.
\eeqsn

We claim that
\beqs
\label{32}
F(x',0)=\sup_{f\in\Lin_S^0}f(x')=:F^0(x').
\eeqs

In the following, we use the convention that when $k\in\R^d$ we put $k'\in\R^{d-1}$ with $k=(k',k_d)$. Note first that Proposition~\ref{prop1} for the sequence $\{a_{(\alpha',0)}\}_{\alpha'\in\N_0^{d-1}}$
implies that
\beqsn
\sup_{f\in\Lin_S^0}f(x')=\sup_{k'\in\R^{d-1}}(\langle k',x'\rangle+h^0_{k'})
\eeqsn
with
\beqsn
h^0_{k'}=\min_{\alpha'\in\N_0^{d-1}}(a_{(\alpha',0)}-\langle k',\alpha'\rangle).
\eeqsn
Moreover, by Proposition~\ref{prop1} once more,
\beqsn
F(x',0)=\sup_{f\in\Lin_S}f(x',0)=\sup_{k\in\R^d}(\langle k',x'\rangle+h_k)
\eeqsn
with
\beqsn
h_k=\min_{\alpha\in\N_0^d}(a_\alpha-\langle k,\alpha\rangle)
\leq\min_{(\alpha',0)\in\N_0^{d-1}\times\{0\}}(a_{(\alpha',0)}-\langle k,(\alpha',0)\rangle)=h^0_{k'}
\eeqsn
and hence
\beqsn
F(x',0)\leq&&\sup_{k\in\R^d}(\langle k',x'\rangle+h^0_{k'})
=\sup_{k'\in\R^{d-1}}(\langle k',x'\rangle+h^0_{k'})\\
=&&\sup_{f\in\Lin_S^0}f(x')=F^0(x').
\eeqsn

In order to prove the opposite inequality, let us now fix $f\in\Lin_S^0$. Then
\beqsn
f(x')=\langle k',x'\rangle+c,\quad k'\in\R^{d-1},\,c\in\R,
\eeqsn
and 
\beqs
\label{new1}
f(\alpha')\leq a_{(\alpha',0)}\quad\forall\alpha'\in\N_0^{d-1}.
\eeqs
We claim that there exists $\Lambda\in\R$ such that
\beqs
\label{claim1}
\tilde f(x):=f(x')+\Lambda x_d\in\Lin_s
\eeqs
for $x=(x',x_d)\in[0,+\infty)^d$.

Indeed, as in Lemma~\ref{lemma1}, from assumption (ii) 
and $|\langle k',\alpha'\rangle+c|\leq\|k'\|\cdot|\alpha|+|c|$
we get that
\beqsn
\langle k',\alpha'\rangle+c>a_\alpha,
\eeqsn
for at most a finite number of points $\alpha=(\alpha',\alpha_d)\in\N_0^d$, with $\alpha_d\geq1$ (because of
\eqref{new1} for $\alpha_d=0$).

There exists then
\beqsn
L:=\min_{\alpha\in\N_0^d}[a_\alpha-(\langle k',\alpha'\rangle+c)].
\eeqsn
Defining 
\beqsn
\Lambda:=\min(0,L),
\eeqsn
we have that $\Lambda\leq0$ and hence for all $\alpha=(\alpha',\alpha_d)\in\N_0^d$ we have two cases:

if $\alpha_d=0$, then by \eqref{new1}
\beqsn
\langle k',\alpha'\rangle+c+\Lambda\alpha_d\leq a_{(\alpha',0)};
\eeqsn

if $\alpha_d\geq1$, then by definition of $L$
\beqsn
\langle k',\alpha'\rangle+c+\Lambda\alpha_d\leq 
\langle k',\alpha'\rangle+c+\Lambda
\leq \langle k',\alpha'\rangle+c+L\leq a_\alpha.
\eeqsn

The two cases above prove \eqref{claim1}.

Setting now
\beqsn
\tilde k:=(k',\Lambda)\in\R^d\quad\mbox{and}\quad
\tilde f(x):=\langle \tilde k,x\rangle+c,
\eeqsn
we have that $\tilde f\in\Lin_S$ and $\tilde f(x',0)=f(x')$.

It follows that
\beqsn
F^0(x')&&=\sup_{f\in\Lin_S^0}f(x')=\sup_{f\in\Lin_S^0}\tilde f(x',0)\\
&&\leq\sup_{f\in\Lin_S}f(x',0)=F(x',0),
\eeqsn
%
and the equality \eqref{32} is therefore proved.

This means that we have reduced the problem to prove that $F^0(x')$ is the maximum convex function
$g:[0,+\infty)^{d-1}\to\R$ such that
\beqsn
g(\alpha')\leq a_{(\alpha',0)}\qquad\forall\alpha'\in\N_0^{d-1}.
\eeqsn
If $x^0=(x_1^0,\ldots,x_{d-1}^0)\in(0,+\infty)^{d-1}$ the thesis follows from \eqref{interior} (in the case of
dimension $d-1$ instead of $d$). If $x^0\in\partial[0,+\infty)^{d-1}$, for instance
$x^0=(x_1^0,\ldots,x_{d-2}^0,0)$ we argue as before and thus reduce to determine the biggest
convex function $g:[0,+\infty)^{d-2}\to\R$ with
\beqsn
g(\alpha_1,\ldots,\alpha_{d-2})\leq a_{(\alpha_1,\ldots,\alpha_{d-2},0,0)},\
\forall(\alpha_1,\ldots,\alpha_{d-2})\in\N_0^{d-2}.
\eeqsn

Proceeding recursively we are finally lead to the one-dimensional case for the sequence
$\{a_{(\alpha_1,0,\ldots,0)}\}_{\alpha_1\in\N_0}$, where the construction of \S\ref{sec2} gives the desired maximum convex function whose graph is below $S$.
Note that on $[0,+\infty)$ the problem of the border does not appear since $a_0\in\R$ by
assumption $(iv)$ and the first line through $(0,a_0)$ is not a vertical line, but the graph of
\beqsn
f_{a_0,k_o}(x)=a_0+k_0x
\eeqsn
defined in \eqref{firstline}.
In particular $F(0)=f_{a_0,k_0}(0)=a_0$.
\end{proof}

\begin{Cor}
\label{cor00}
Given a sequence $\{a_\alpha\}_{\alpha\in\N^d_0}$ satisfying $(i)-(iv)$, its convex minorant sequence $\{a^c_\alpha\}_{\alpha\in\N^d_0}$ is defined by
\beqsn
a^c_\alpha=F(\alpha),
\eeqsn
for $F$ as in \eqref{defF}, or equivalently by
\beqs
\label{001}
a^c_\alpha:=\sup_{k\in\R^d}(\langle k,\alpha\rangle+h_k),
\eeqs
for
\beqs
\label{002}
h_k=\inf_{\alpha\in\N_0^d}(a_\alpha-\langle k,\alpha\rangle).
\eeqs
In particular $a_0^c=a_0$.
\end{Cor}

%
%

\section{Construction of the log-convex minorant}
\label{sec4}

Let $\{M_\alpha\}_{\alpha\in\N_0^d}$ be a sequence of positive real numbers such that
\beqs
\label{limMalpha}
\lim_{|\alpha|\to+\infty}M_\alpha^{1/|\alpha|}=+\infty
\eeqs
(we can also allow $M_\alpha=+\infty$ for finitely many multi-indices $\alpha\neq0$).
We say that the sequence $\mathbf{M}=\{M_\alpha\}_{\alpha\in\N_0^d}$ is
{\em normalized} if $M_0=1$.

For a normalized sequence $\mathbf{M}=\{M_\alpha\}_{\alpha\in\N_0^d}$ we define
its {\em associated function} $\omega_{\mathbf{M}}$ by
\beqsn
\omega_{\mathbf{M}}(t):=\sup_{\alpha\in\N_{0,t}^d}\log\frac{|t^\alpha|}{M_\alpha},
\qquad t\in\R^d,
\eeqsn
with
\beqsn
\N_{0,t}^d:=\{\alpha\in\N_0^d:\ \alpha_j=0\ \mbox{if}\ t_j=0,\, j=1,\ldots,d\},
\eeqsn
and the convention that $0^0=1$. We observe that here the supremum is made on $\N_{0,t}^d$ in order to ensure that in the definition of the associated function the argument of the logarithm is not $0$; we can equivalently write
\beqsn
\omega_{\mathbf{M}}(t)=\sup_{\alpha\in\N^d_0}\log\frac{|t^\alpha|}{M_\alpha},
\qquad t\in\R^d,
\eeqsn
with the convention that $\log 0=-\infty$. In what follows (in particular in Section \ref{sec5}) we use the latter expression for convenience.

Condition \eqref{limMalpha} ensures that $\omega_{\bfM}(t)<+\infty$ for all $t\in\R^d$
(see  \cite[Chap.~I]{M} or \cite[Rem.~1]{BJOS-Banach}).

The function $\omega_{\mathbf{M}}(t)$ is increasing on $(0,+\infty)^d$ in the following sense: 
$\omega_{\mathbf{M}}(t)\leq \omega_{\mathbf{M}}(s)$ if $t\leq s$ with the order relation
$t_j\leq s_j$ for all $1\leq j\leq d$.

Consider then
\beqs
\label{aalpha}
a_\alpha=\log M_\alpha,\qquad\alpha\in\N_0^d.
\eeqs
By \eqref{limMalpha} the sequence $\{a_\alpha\}_{\alpha\in\N_0^d}$ satisfies
all assumptions $(i)-(iv)$ of \S\ref{sec1}.

We can thus consider the convex minorant sequence $\{a_\alpha^c\}_{\alpha\in\N_0^d}$
as in \eqref{001}-\eqref{002} and call
\beqsn
M^{\lc}_\alpha:=\exp a^c_\alpha,\qquad\forall\alpha\in\N_0^d,
\eeqsn
the {\em log-convex minorant} of $\{M_\alpha\}_{\alpha\in\N_0^d}$.

By the results of the previous section $\{\log M^{\lc}_\alpha\}_{\alpha\in\N_0^d}$ is the largest
convex sequence less than or equal to $\{a_\alpha\}_{\alpha\in\N_0^d}$ defined in \eqref{aalpha}
and they coincide if $\{M_\alpha\}_{\alpha\in\N^d_0}$ is log-convex, according to the following
\begin{Def}
\label{defcvx}
A sequence $\{M_\alpha\}_{\alpha\in\N^d_0}$ is said to be {\em log-convex} if there exists
a convex function $F:[0,+\infty)^d\to\R$ with $F(\alpha)=\log M_\alpha$ for all $\alpha\in\N_0^d$.
\end{Def}

By construction
\beqs
\label{41}
a^c_\alpha=\sup_{k\in\R^d}(\langle k,\alpha\rangle+h_k)
=\sup_{k\in\R^d}(\langle k,\alpha\rangle-A(k)),
\eeqs
where the so-called {\em trace function} $A(k)$ is given by
\beqs
\nonumber
A(k)=-h_k=&&-\inf_{\alpha\in\N_0^d}(a_\alpha-\langle k,\alpha\rangle)\\
\nonumber
=&&\sup_{\alpha\in\N_0^d}(\langle k,\alpha\rangle-a_\alpha)\\
\nonumber
=&&\sup_{\alpha\in\N_0^d}(\langle k,\alpha\rangle-\log M_\alpha)\\
\nonumber
=&&\sup_{\alpha\in\N_0^d}\log\frac{e^{\langle k,\alpha\rangle}}{M_\alpha}\\
\nonumber
=&&\sup_{\alpha\in\N_0^d}\log\frac{|(e^k)^\alpha|}{M_\alpha}\\
\label{42}
=&&\omega_{\mathbf{M}}(e^k),
\eeqs
since
\beqsn
|(e^k)^\alpha|=|(e^{k_1},\ldots,e^{k_d})^\alpha|=|e^{k_1\alpha_1}\cdots e^{k_d\alpha_d}|
=e^{\langle k,\alpha\rangle}.
\eeqsn

From \eqref{41} and \eqref{42} we have that
\beqs
\nonumber
M^{\lc}_\alpha=&&\exp a^c_\alpha
=\exp\{\sup_{k\in\R^d}(\langle k,\alpha\rangle-\omega_{\mathbf{M}}(e^k))\}\\
\nonumber
=&&\sup_{k\in\R^d}\frac{|(e^k)^\alpha|}{\exp\omega_{\mathbf{M}}(e^k)}\\
\nonumber
=&&\sup_{s\in(0,+\infty)^d}\frac{|s^\alpha|}{\exp\omega_{\mathbf{M}}(s)}\\
\label{Mlc}
=&&\sup_{s\in(0,+\infty)^d}\frac{s^\alpha}{\exp\omega_{\mathbf{M}}(s)}.
\eeqs

In particular, being $M_\alpha^{\lc}\leq M_\alpha$ for any normalized sequence of positive real numbers
$\{M_\alpha\}_{\alpha\in\N_0^d}$, we have 
\beqs
\label{55bis}
\sup_{s\in(0,+\infty)^d}\frac{s^\alpha}{\exp\omega_{\mathbf{M}}(s)}\leq M_\alpha,
\qquad\forall\alpha\in\N_0^d,
\eeqs
 and, moreover, $\{M_\alpha\}_{\alpha\in\N_0^d}$
is log-convex if and only if the equality holds. Hence, we have proved:
\begin{Th}
\label{cor1}
Let $\{M_\alpha\}_{\alpha\in\N_0^d}$ be a normalized sequence of positive real numbers satisfying \eqref{limMalpha}. Then, the sequence $\{M_\alpha\}_{\alpha\in\N_0^d}$ is log-convex if and only if
\beqs
\label{Q3}
M_\alpha=\sup_{s\in(0,+\infty)^d}\frac{s^\alpha}{\exp\omega_{\mathbf{M}}(s)}, 
\qquad\forall\alpha\in\N_0^d.
\eeqs
\end{Th}

\begin{Rem}
\label{rem4}
\begin{em}
Note that if $\{M_\alpha\}_{\alpha\in\N^d_0}$ is a log-convex sequence, i.e.
$\log M_\alpha=F(\alpha)$ for a convex function $F:[0,+\infty)^d\to\R$, 
and $e_j$ denotes the $j$-th element of the canonical basis of $\R^d$ with all entries equal to 0 except the $j$-th entry equal to 1, then
\beqsn
F(\alpha)=F\left(\frac12(\alpha-e_j)+\frac12(\alpha+e_j)\right)
\leq\frac12 F(\alpha-e_j)+\frac12 F(\alpha+e_j),
\eeqsn
that is
\beqsn
2\log M_\alpha\leq \log M_{\alpha-e_j}+\log M_{\alpha+e_j}.
\eeqsn
This yields the {\em coordinate-wise log-convexity condition} for a sequence
$\{M_\alpha\}_{\alpha\in\N_0^d}$\,:
\beqs
\label{LOG-C}
M_\alpha^2\leq M_{\alpha-e_j}M_{\alpha+e_j},
\qquad\alpha\in\N_0^d,\   1\leq j \leq d,\  \alpha_j\geq1.
\eeqs
This condition is clearly weaker than the condition of logarithmic convexity given in Definition~\ref{defcvx} since there are functions which are coordinate-wise convex but not
convex as functions of more variables.
In particular, from \eqref{Mlc} we have that \eqref{LOG-C} is not sufficient to obtain \eqref{Q3}; for an explicit example, see Example~\ref{not-convex} below.
Clearly in the one-dimensional case the two notions of log-convexity coincide, and \eqref{Q3}
was already known (see \cite[Prop.~3.2]{K}).
\end{em}
\end{Rem}

\begin{Ex}\label{not-convex}{\rm 
The function of two variables $F(x,y)=(x+1)^2(y+1)^2$ is coordinate-wise convex but not convex in $[0,+\infty)^2$, and the sequence defined by $M_\alpha:=e^{F(\alpha)-1}$ is normalized and satisfies \eqref{limMalpha}, but does not satisfy \eqref{Q3}.
}
\end{Ex}

\section{Characterization of inclusion relations of spaces of rapidly decreasing ultradifferentiable functions}
\label{sec5}

Let us first recall the notion of weight matrices (anisotropic framework) as considered in
\cite[Sec.~3]{BJOS-Banach}.

A {\em weight matrix} $\calM$ is the set
\beqsn
\calM:=\{(\bfM^{(\lambda)})_{\lambda>0}: \bfM^{(\lambda)}=
(M_\alpha^{(\lambda)})_{\alpha\in\N_0^d}, M_0^{(\lambda)}=1,
M_\alpha^{(\lambda)}\leq M_\alpha^{(\kappa)}\ \forall\alpha\in\N_0^d\ \forall0<\lambda\leq\kappa\}.
\eeqsn

Denoting by $\|\cdot\|_{\infty}$ the supremum norm we consider the following spaces
of matrix weighted global ultradifferentiable functions of Roumieu/Beurling type
\beqsn
&&\Sch_{\{\calM\}}(\R^d):=\{f\in C^\infty(\R^d):\ \exists \lambda,h,C>0:\
\sup_{\alpha,\beta\in\N_0^d}\frac{\|x^\alpha\partial^\beta f\|_\infty}{h^{|\alpha+\beta|}M_{\alpha+\beta}^{(\lambda)}}\leq C\}\\
&&\Sch_{(\calM)}(\R^d):=\{f\in C^\infty(\R^d):\ \forall \lambda,h>0\,\exists C_{h,\lambda}>0:\
\sup_{\alpha,\beta\in\N_0^d}\frac{\|x^\alpha\partial^\beta f\|_\infty}{h^{|\alpha+\beta|}M_{\alpha+\beta}^{(\lambda)}}\leq C_{h,\lambda}\}
\eeqsn
endowed with the inductive limit topology in the Roumieu case and the projective limit topology
in the Beurling case (see \cite[Sec.~3]{BJOS-Banach}).

In order to characterize the inclusion of spaces of this type, given two weight matrices
$\calM=\{(\bfM^{(\lambda)})_{\lambda>0}\}$ and
$\calN=\{(\bfN^{(\lambda)})_{\lambda>0}\}$ we introduce the following relations:
\beqsn
&&\calM\{\preceq\}\calN \quad\mbox{if}\quad
\forall\lambda>0\exists\kappa>0\exists C\geq1\,\mbox{s.t.}\ M_\alpha^{(\lambda)}
\leq C^{|\alpha|}N_\alpha^{(\kappa)}\ \forall\alpha\in\N_0^d\\
&&\calM(\preceq)\calN \quad\mbox{if}\quad
\forall\lambda>0\exists\kappa>0\exists C\geq1\,\mbox{s.t.}\ M_\alpha^{(\kappa)}
\leq C^{|\alpha|}N_\alpha^{(\lambda)}\ \forall\alpha\in\N_0^d\\
&&\calM\vartriangleleft\calN \quad\mbox{if}\quad
\forall\lambda>0\forall\kappa>0\forall h>0\exists C\geq1\,\mbox{s.t.}\ M_\alpha^{(\lambda)}
\leq Ch^{|\alpha|}N_\alpha^{(\kappa)}\ \forall\alpha\in\N_0^d.
\eeqsn

We shall also need the following conditions for a weight matrix
$\calM=\{(\bfM^{(\lambda)})_{\lambda>0}\}$ (see \cite{BJOS-Banach}; cf. also \cite{L}), in the
Roumieu setting
\beqs
\label{L12R}
&&\forall\lambda>0 \exists\kappa\geq\lambda, B, C, H>0:\quad \alpha^{\alpha/2}M_\beta^{(\lambda)}\leq BC^{|\alpha|}H^{|\alpha+\beta|}M_{\alpha+\beta}^{(\kappa)}\  \
\forall\alpha,\beta\in\N_0^d\\
\label{L21R}
&&\forall\lambda>0 \exists\kappa\geq\lambda, A\geq1:\quad
M_{\alpha+e_j}^{(\lambda)}\leq A^{|\alpha|+1}M_{\alpha}^{(\kappa)}\ \
\forall\alpha\in\N_0^d,1\leq j\leq d\\
\label{L37R}
&&\forall\lambda>0 \exists\kappa\geq\lambda, A\geq1:\quad
M_\alpha^{(\lambda)}M_\beta^{(\lambda)}\leq A^{|\alpha+\beta|}M_{\alpha+\beta}^{(\kappa)}\ \
\forall\alpha,\beta\in\N_0^d
\eeqs
and in the Beurling case
\beqs
\nonumber
&&\forall\lambda>0 \exists0<\kappa\leq\lambda, H>0:\forall C>0\exists B>0:\\
\label{L12B}
&&\hspace*{2cm} \alpha^{\alpha/2}M_\beta^{(\kappa)}\leq BC^{|\alpha|}H^{|\alpha+\beta|}M_{\alpha+\beta}^{(\lambda)}\ \
\forall\alpha,\beta\in\N_0^d\\
\nonumber
&&\forall\lambda>0 \exists0<\kappa\leq\lambda, A\geq1:\\
\label{L21B}
&&\hspace*{2cm} M_{\alpha+e_j}^{(\kappa)}\leq A^{|\alpha|+1}M_{\alpha}^{(\lambda)}\ \
\forall\alpha\in\N_0^d,1\leq j\leq d\\
\nonumber
&&\forall\lambda>0 \exists0<\kappa\leq\lambda, A\geq1:\\
\label{63B}
&&\hspace*{2cm}M_\alpha^{(\kappa)}M_\beta^{(\kappa)}\leq A^{|\alpha+\beta|}M_{\alpha+\beta}^{(\lambda)}\ \
\forall\alpha,\beta\in\N_0^d.
\eeqs

Note that \eqref{L12R} for $\beta=0$ implies that $\bfM^{(\kappa)}$ satisfies
\eqref{limMalpha} for some $\kappa>0$, and hence for all $\kappa'\geq\kappa$.
Similarly \eqref{L12B}  implies that $\bfM^{(\lambda)}$ satisfies
\eqref{limMalpha} for all $\lambda>0$ (see \cite[Rem.~3]{BJOS-Banach}).
Since condition \eqref{limMalpha} ensures that the associate weight function $\omega_\bfM$
is finite, the above remarks are essential to recall, from \cite[Thm.~1]{BJOS-Banach},
that if the weight matrix $\calM=\{(\bfM^{(\lambda)})_{\lambda>0}\}$ satisfies
\eqref{L12R} and \eqref{L21R} then
$\Sch_{\{\calM\}}$ is isomorphic to the sequence space
\beqs
\label{LLR}
\Lambda_{\{\calM\}}:=\{\bfc=(c_\alpha)_{\alpha\in\N_0^d}\in\C^{\N_0^d}:\ \exists\lambda,h>0\,\mbox{s.t.}\,
\sup_{\alpha\in\N_0^d}|c_\alpha|e^{\omega_{\bfM^{(\lambda)}}(\alpha^{1/2}/h)}<+\infty\},
\eeqs
and similarly in the Beurling case, for a weight matrix $\calM$ having \eqref{L12B}-\eqref{L21B}, the space
$\Sch_{(\calM)}$ is isomorphic to
\beqs
\label{LLB}
\Lambda_{(\calM)}:=\{\bfc=(c_\alpha)_{\alpha\in\N_0^d}\in\C^{\N_0^d}:\ \forall\lambda,h>0\,
\sup_{\alpha\in\N_0^d}|c_\alpha|e^{\omega_{\bfM^{(\lambda)}}(\alpha^{1/2}/h)}<+\infty\},
\eeqs
where $\alpha^{1/2}:=(\alpha_1^{1/2},\ldots,\alpha_d^{1/2})$.

It will also be useful, in the following, to write the sequence spaces \eqref{LLR}-\eqref{LLB}
as follows (see \cite{BJOS-Banach}):
\beqs
\label{LambdaR}
&&\Lambda_{\{\calM\}}=\{\bfc=(c_\alpha)_{\alpha\in\N_0^d}\in\C^{\N_0^d}:\ \exists j\in\N\,\mbox{s.t.}\,
\sup_{\alpha\in\N_0^d}|c_\alpha|e^{\omega_{\bfM^{(j)}}(\alpha^{1/2}/j)}<+\infty\}\\
\label{LambdaB}
&&\Lambda_{(\calM)}=\{\bfc=(c_\alpha)_{\alpha\in\N_0^d}\in\C^{\N_0^d}:\ \forall j\in\N\,
\sup_{\alpha\in\N_0^d}|c_\alpha|e^{\omega_{\bfM^{(1/j)}}(j\alpha^{1/2})}<+\infty\}.
\eeqs

We say that a weight matrix $\calM=\{(\bfM^{(\lambda)})_{\lambda>0}\}$ is
{\em log-convex} if $\{M^{(\lambda)}_\alpha\}_{\alpha\in\N_0^d}$ is log-convex (according to Definition~\ref{defcvx}) for all $\lambda>0$.



Let us remark that in the one-dimensional case the assumption of log-convexity,
together with $M_0=1$, implies both \eqref{L37R} and \eqref{63B} with $A=1$ (we are here considering the sequence case for simplicity, i.e., the case when in the weight matrix $\calM$ all the $\bfM^{(\lambda)}$, $\lambda>0$, coincide) since the convex function $F(p)=\log M_p$ has increasing difference quotient and hence
\beqsn
\frac{\log M_{p+q}-\log M_p}{q}\geq\frac{\log M_q-\log M_0}{q}=\frac{\log M_q}{q}.
\eeqsn
On the contrary, in the more-dimensional case log-convexity and $M_0=1$
do not imply \eqref{L37R}/\eqref{63B}, not even under the additional conditions 
\eqref{L12R}/\eqref{L12B} and \eqref{L21R}/\eqref{L21B}.
Take, for instance, $M_\alpha=\alpha^{\alpha/2}e^{\max\{\alpha_1^2,\alpha_2^2\}}$
for $\alpha=(\alpha_1,\alpha_2)\in\N_0^2$ (with the convention $0^0:=1$).
It is easy to check that it is log-convex (since $\log M_\alpha$ is the sum of convex functions) and satisfies \eqref{L12R} and \eqref{L21R}.
However, taking $\alpha=(n,0)$ and $\beta=(0,n)$ condition  \eqref{L37R} is not valid for
$n\to+\infty$.

The above remarks explain why conditions  \eqref{L37R} and \eqref{63B}, that we shall need in 
Theorem~\ref{th42}, are not required in the one-dimensional/isotropic case
(see \cite{BJOS-fuzzy}).

\begin{Th}
\label{th42}
Let $\calM=\{(\bfM^{(\lambda)})_{\lambda>0}\}$ and $\calN=\{(\bfN^{(\lambda)})_{\lambda>0}\}$
be two weight matrices and assume that $\calM$ is log-convex.
Then the following characterizations hold:
\begin{enumerate}[(i)]
\item
Let $\calM$ satisfy \eqref{L12R}-\eqref{L37R} and $\calN$ satisfy \eqref{L12R}-\eqref{L21R}. Then:
$\Sch_{\{\calM\}}(\R^d)\subseteq\Sch_{\{\calN\}}(\R^d)$ holds with continuous inclusion if and only if $\calM\{\preceq\}\calN$;
\item
Let $\calM$  satisfy \eqref{L12R}-\eqref{L37R} and $\calN$  satisfy \eqref{L12B}-\eqref{L21B}. Then:
$\Sch_{\{\calM\}}(\R^d)\subseteq\Sch_{(\calN)}(\R^d)$ holds with continuous inclusion if and only if $\calM\vartriangleleft\calN$;
\item
Let $\calM$  satisfy \eqref{L12B}-\eqref{63B} and $\calN$ satisfy \eqref{L12B}-\eqref{L21B}. Then:
$\Sch_{(\calM)}(\R^d)\subseteq\Sch_{(\calN)}(\R^d)$ holds with continuous inclusion if and only if
$\calM(\preceq)\calN$.
\end{enumerate}
\end{Th}

\begin{proof}
The sufficiency of condition $\calM\{\preceq\}\calN$ (resp.
$\calM\vartriangleleft\calN$, $\calM(\preceq)\calN$) for the inclusion of the spaces clearly follows from the definition of the spaces (hypotheses \eqref{L12R}-\eqref{L37R}/\eqref{L12B}-\eqref{63B}
and log-convexity of $\calM$ are not needed at this step).

We now prove the necessity of the conditions.

$(i)$: Let us assume that $\Sch_{\{\calM\}}(\R^d)\subseteq\Sch_{\{\calN\}}(\R^d)$ holds with continuous inclusion and prove that $\calM\{\preceq\}\calN$.

By the already mentioned isomorphism with the sequence space \eqref{LambdaR}, we have that
\beqs
\label{Q1}
\Lambda_{\{\calM\}}\simeq\Sch_{\{\calM\}}(\R^d)\subseteq\Sch_{\{\calN\}}(\R^d)
\simeq\Lambda_{\{\calN\}}.
\eeqs

For some fixed (but arbitrary) $j\in\N$, $j\geq2$, let us consider
the sequence $\bfc:=(c_\alpha)_{\alpha\in\N_0^d}\in\C^{\N_0^d}$
defined by
\beqs
\label{Q5}
c_\alpha:=e^{-\omega_{\bfM^{(j)}}(\bar\alpha^{1/2}/j)}
\eeqs
for $\bar\alpha=(\bar\alpha_1,\ldots,\bar\alpha_d)$ with
\beqs
\label{alphakbar}
\bar\alpha_k=\begin{cases}
\alpha_k & \mbox{if}\ \alpha_k\neq1\cr
0 & \mbox{if}\ \alpha_k=1,
\end{cases}
\eeqs
with the convention $e^{-\infty}=0$.

Let us prove that $\bfc\in\Lambda_{\{\calM\}}$.
If $\omega_{\bfM^{(j)}}(\bar{\alpha}^{1/2}/j)=+\infty$ there is nothing to prove.
If $\omega_{\bfM^{(j)}}(\bar{\alpha}^{1/2}/j)<+\infty$ we must find
 $h\in\N$ and $C>0$ such that
\beqsn
|c_\alpha|=e^{-\omega_{\bfM^{(j)}}(\bar\alpha^{1/2}/j)}\leq C
e^{-\omega_{\bfM^{(h)}}(\alpha^{1/2}/h)}\qquad
\forall\alpha\in\N_0^d,
\eeqsn
i.e.
\beqs
\label{QQ3}
\omega_{\bfM^{(h)}}\left(\frac{\sqrt{\alpha_1}}{h},\ldots,\frac{\sqrt{\alpha_d}}{h}\right)
\leq\log C+\omega_{\bfM^{(j)}}\left(\frac{\sqrt{\bar\alpha_1}}{j},\ldots,\frac{\sqrt{\bar\alpha_d}}{j}\right).
\eeqs
Let us first choose $\bar h\in\N$, $\bar h\geq j$, such that $\bfM^{(h)}$ satisfies \eqref{limMalpha},
and hence
$\omega_{\bfM^{(h)}}(t)<+\infty$ for all $t\in\R^d$,
for all $ h\geq\bar h$ (this is possible because of assumption \eqref{L12R}).
Then by assumption \eqref{L37R} we choose $h\in\N$, $h\geq\bar h$, and $A\geq1$ such that
\beqs
\label{jh}
M^{(j)}_\alpha M^{(j)}_\beta\leq A^{|\alpha+\beta|}M^{(h)}_{\alpha+\beta},\qquad
\forall\alpha,\beta\in\N_0^d.
\eeqs
For such a choice of $h$ and from the definition of associated function
\beqs
\label{QQ1}
\omega_{\bfM^{(h)}}\left(\frac{\sqrt{\alpha_1}}{h},\ldots,\frac{\sqrt{\alpha_d}}{h}\right)
=\sup_{\beta_1,\ldots,\beta_d\in\N_0}\log
\frac{\alpha_1^{\beta_1/2}\cdots\alpha_d^{\beta_d/2}}{h^{\beta_1+\ldots+\beta_d}
M^{(h)}_{(\beta_1,\ldots,\beta_d)}}.
\eeqs

Let us now assume, without loss of generality, that the entries of $\alpha$ equal to 1 (if there are some) are in the first positions, i.e.
$(\alpha_1,\ldots,\alpha_d)=(1,\ldots,1,\alpha_{s+1},\ldots,\alpha_d)$, for some $s\geq0$, with
$\alpha_k\neq1$ for $s+1\leq k\leq d$.
By  \eqref{jh} we have that
\beqsn
M^{(h)}_{(\beta_1,\ldots,\beta_d)}\geq
\frac{1}{A^{\beta_1+\ldots+\beta_d}}M^{(j)}_{(\beta_1,\ldots,\beta_s,0,\ldots,0)}
M^{(j)}_{(0,\ldots,0,\beta_{s+1},\ldots,\beta_d)}
\eeqsn
and hence from \eqref{QQ1}
\beqsn
\omega_{\bfM^{(h)}}\left(\frac{\alpha^{1/2}}{h}\right)\leq&&
\sup_{\beta_1,\ldots,\beta_d\in\N_0}\Bigg(\log
\frac{\alpha_{s+1}^{\beta_{s+1}/2}\cdots\alpha_d^{\beta_d/2}}{\left(
\frac{h}{A}\right)^{\beta_{s+1}+\ldots+\beta_d}
M^{(j)}_{(0,\ldots,0,\beta_{s+1},\ldots,\beta_d)}}\\
&&
\qquad+\log\frac{(A/h)^{\beta_1+\ldots+\beta_s}}{
M^{(j)}_{(\beta_{1},\ldots,\beta_s,0,\ldots,0)}}\Bigg)\\
\leq&&
\sup_{\beta_{s+1},\ldots,\beta_d\in\N_0}\log
\frac{\alpha_{s+1}^{\beta_{s+1}/2}\cdots\alpha_d^{\beta_d/2}}{\left(
\frac{h}{A}\right)^{\beta_{s+1}+\ldots+\beta_d}
M^{(j)}_{(0,\ldots,0,\beta_{s+1},\ldots,\beta_d)}}\\
&&+\sup_{\beta_{1},\ldots,\beta_s\in\N_0}\log\frac{\left(\frac Ah\right)^{\beta_1+\ldots+\beta_s}}{
M^{(j)}_{(\beta_{1},\ldots,\beta_s,0,\ldots,0)}}.
\eeqsn

Eventually enlarging $h$ so that $h\geq Aj$ we thus have
\beqs
\nonumber
\omega_{\bfM^{(h)}}\left(\frac{\alpha^{1/2}}{h}\right)\leq&&
\sup_{\beta_{s+1},\ldots,\beta_d\in\N_0}\log
\frac{\alpha_{s+1}^{\beta_{s+1}/2}\cdots\alpha_d^{\beta_d/2}}{j^{\beta_{s+1}+\ldots+\beta_d}
M^{(j)}_{(0,\ldots,0,\beta_{s+1},\ldots,\beta_d)}}+\omega_{\bar\bfM^{(j)}}\left(\frac Ah\right)\\
\label{finiteC}
\leq&&\omega_{\bfM^{(j)}}\left(\frac{\bar\alpha^{1/2}}{j}\right)+C_{j,h}
\eeqs
for $\{\bar M_\beta^{(j)}\}_{\beta\in\N_0^s}$ the sequence given by $\bar M_\beta^{(j)}=M^{(j)}_{(\beta_{1},\ldots,\beta_s,0,\ldots,0)}$
and $C_{j,h}=\omega_{\bar\bfM^{(j)}}(A/h)$.
Note that this constant is finite because if $|(\beta_1,\ldots,\beta_s)|\to+\infty$ then also
$|(\beta_1,\ldots,\beta_s,0,\ldots,0)|\to+\infty$ and hence
\beqsn
\left(\bar{M}^{j}_\beta\right)^{1/|\beta|}=\left({M}^{j}_{(\beta_1,\ldots,\beta_s,0,\ldots,0)}\right)^{1/(\beta_1+\cdots+\beta_s)}\longrightarrow+\infty
\eeqsn
by \eqref{limMalpha}.

Inequality \eqref{finiteC} proves \eqref{QQ3} and hence $\bfc\in\Lambda_{\{\calM\}}$.

From \eqref{Q1} we have that
$\bfc\in\Lambda_{\{\calN\}}$, and therefore there exist $\ell\in\N$ and $C\geq1$ such that
\beqsn
e^{-\omega_{\bfM^{(j)}}(\bar\alpha^{1/2}/j)}=|c_\alpha|
\leq C e^{-\omega_{\bfN^{(\ell)}}(\alpha^{1/2}/\ell)},\qquad \forall\alpha\in\N_0^d,
\eeqsn
i.e.
\beqs
\label{Q2}
\omega_{\bfN^{(\ell)}}(\alpha^{1/2}/\ell)\leq
\log C+\omega_{\bfM^{(j)}}(\bar\alpha^{1/2}/j),\qquad \forall\alpha\in\N_0^d.
\eeqs

Fix now $t\in(0,+\infty)^d$ and set $\mathbf1:=(1,\ldots,1)\in\R^d$. There exists $\alpha\in\N_0^d$ with $\alpha^{1/2}<t\leq
(\alpha+\mathbf1)^{1/2}$, i.e. $\alpha_k<t_k^2\leq\alpha_k+1$ for all $1\leq k\leq d$.
From \eqref{Q2} we have
\beqsn
\omega_{\bfN^{(\ell)}}(t/\ell)\leq&&\omega_{\bfN^{(\ell)}}((\alpha+\mathbf1)^{1/2}/\ell)
\leq\log C+\omega_{\bfM^{(j)}}((\overline{\alpha+\mathbf1})^{1/2}/j)\\
\leq&&\log C+\omega_{\bfM^{(j)}}(\alpha^{1/2})
\leq\log C+\omega_{\bfM^{(j)}}(t)
\eeqsn
because $\omega_{\bfM^{(j)}}$ is increasing on $(0,+\infty)^d$ and $(\overline{\alpha+\mathbf1})^{1/2}/j\leq\alpha^{1/2}$
since $\overline{\alpha_k+1}=0$ if $\alpha_k=0$ and $(\overline{\alpha_k+1})^{1/2}/j=\sqrt{\alpha_k+1}/j\leq\sqrt{\alpha_k}$ if
$\alpha_k\geq1$ and $j\geq2$.

From  \eqref{55bis} and \eqref{Q3} we thus have that
\beqsn
M^{(j)}_\alpha=&&\sup_{t\in(0,+\infty)^d}\frac{t^\alpha}{\exp\omega_{\bfM^{(j)}}(t)}
\leq C\sup_{t\in(0,+\infty)^d}\frac{t^\alpha}{\exp\omega_{\bfN^{(\ell)}}(t/\ell)}\\
=&&C\sup_{s\in(0,+\infty)^d}\frac{(s\ell)^\alpha}{\exp\omega_{\bfN^{(\ell)}}(s)}
\leq C\ell^{|\alpha|}N_\alpha^{(\ell)}.
\eeqsn

We have thus proved that
\beqsn
\forall j\in\N, j\geq2, \exists\ell\in\N, C\geq1:\quad
M^{(j)}_\alpha\leq C\ell^{|\alpha|}N_\alpha^{(\ell)}\quad \forall\alpha\in\N_0^d.
\eeqsn
Since $M_\alpha^{(j)}\leq M_\alpha^{(2)}$ for $j<2$ and the sequences in the weight matrices are normalized, we have proved that
$\calM\{\preceq\}\calN$.

$(ii)$:
Assuming
\beqsn
\Lambda_{\{\calM\}}\simeq\Sch_{\{\calM\}}(\R^d)\subseteq\Sch_{(\calN)}(\R^d)
\simeq\Lambda_{(\calN)}
\eeqsn
we have to prove that $\calM\vartriangleleft\calN$.

We choose $\bfc=(c_\alpha)_{\alpha\in\N_0^d}\in\C^{\N^d}$ as in \eqref{Q5} so that
$\bfc\in\Lambda_{\{\calM\}}\subseteq\Lambda_{(\calN)}(\R^d)$ i.e.
\beqs
\label{Q6}\qquad\
\forall j\in\N,j\geq2,\forall\ell\in\N\, \exists C\geq1:\ \
e^{-\omega_{\bfM^{(j)}}(\bar\alpha^{1/2}/j)}=|c_\alpha|\leq
Ce^{-\omega_{\bfN^{(1/\ell)}}(\alpha^{1/2}\ell)}\
\forall\alpha\in\N_0^d.
\eeqs

As in the case $(i)$, for any $t\in(0,+\infty)^d$ we can choose
$\alpha\in\N_0^d$ with $\alpha^{1/2}<t\leq (\alpha+\mathbf1)^{1/2}$
so that from \eqref{Q6}:
\beqsn
\omega_{\bfN^{(1/\ell)}}(t\ell)\leq&&\omega_{\bfN^{(1/\ell)}}((\alpha+\mathbf1)^{1/2}\ell)
\leq\log C+\omega_{\bfM^{(j)}}((\overline{\alpha+\mathbf1})^{1/2}/j)\\
\leq&&\log C+\omega_{\bfM^{(j)}}(\alpha^{1/2})
\leq\log C+\omega_{\bfM^{(j)}}(t).
\eeqsn

It follows that for all $j\in\N$, $j\geq2$,  and for all $\ell\in\N$ there exists $C\geq1$ such that
\beqsn
M^{(j)}_\alpha=&&\sup_{t\in(0,+\infty)^d}\frac{t^\alpha}{\exp\omega_{\bfM^{(j)}}(t)}
\leq C\sup_{t\in(0,+\infty)^d}\frac{t^\alpha}{\exp\omega_{\bfN^{(1/\ell)}}(t\ell)}\\
=&&C\sup_{s\in(0,+\infty)^d}\frac{(s/\ell)^\alpha}{\exp\omega_{\bfN^{(1/\ell)}}(s)}
\leq C\frac{1}{\ell^{|\alpha|}}N_\alpha^{(1/\ell)},
\eeqsn
because of \eqref{55bis} and \eqref{Q3}.
Since $M_\alpha^{(j)}\leq M_\alpha^{(2)}$ for $j<2$ we finally obtain
\beqs
\label{Q10}
\forall j\in\N\, \forall\ell\in\N\, \exists C\geq1:\quad
M_\alpha^{(j)}\leq C \frac{1}{\ell^{|\alpha|}}N_\alpha^{(1/\ell)}\ \
\forall\alpha\in\N_0^d.
\eeqs
Now, it is obvious that \eqref{Q10} implies condition $\calM\vartriangleleft\calN$. In fact, it is enough to take $\ell\ge \max\{1/\kappa ,1/h\}$ for any given $\kappa, h>0$ in the definition of $\calM\vartriangleleft\calN$.

%

$(iii)$:
Assuming
\beqsn
\Lambda_{(\calM)}\simeq\Sch_{(\calM)}(\R^d)\subseteq\Sch_{(\calN)}(\R^d)
\simeq\Lambda_{(\calN)}
\eeqsn
we have to prove that $\calM(\preceq)\calN$.

By the continuity of the inclusion and \eqref{LambdaB} we have 
\beqs
\label{Q8}
\begin{split}
&\forall \ell\in\N\,\exists h\in\N, C\geq1\,\mbox{s.t.}\,\forall\bfc\in\Lambda_{(\calM)}\\
&\sup_{\alpha\in\N_0^d}|c_\alpha|e^{\omega_{\bfN^{(1/\ell)}}(\ell\alpha^{1/2})}\leq C
\sup_{\alpha\in\N_0^d}|c_\alpha|e^{\omega_{\bfM^{(1/h)}}(h\alpha^{1/2})}.
\end{split}
\eeqs

Let us now fix $\alpha\in\N_0^d$ and assume as before, without loss of generality, that 
$(\alpha_1,\ldots,\alpha_d)=(1,\ldots,1,\alpha_{s+1},\ldots,\alpha_d)$, for some $s\geq0$, with
$\alpha_k\neq1$ for $s+1\leq k\leq d$. Setting $\bar\alpha$ as in \eqref{alphakbar},
by \eqref{63B} there exist $j\in\N$, $j\geq h$, and $A\geq1$ such that
\beqs
\nonumber
\omega_{\mathbf{M}^{(1/h)}}(h\alpha^{1/2})=&&\sup_{\beta\in\N_0^d}\log
\frac{h^{\beta_1+\ldots+\beta_d}\alpha_1^{\beta_1/2}\cdots\alpha_d^{\beta_d/2}}{M^{(1/h)}_\beta}\\
\nonumber
\leq&&\sup_{\beta\in\N_0^d}\log
\frac{h^{\beta_1+\ldots+\beta_d}\cdot1\cdot\alpha_{s+1}^{\beta_{s+1}/2}\cdots
\alpha_d^{\beta_d/2}}{(1/A)^{\beta_1+\ldots+\beta_d}M^{(1/j)}_{(\beta_1,\ldots,\beta_s,0,\ldots0)}
M^{(1/j)}_{(0,\ldots,0,\beta_{s+1},\ldots,\beta_d)}}\\
\nonumber
\leq&&\sup_{\beta\in\N_0^d}\log
\frac{(Ah)^{\beta_1+\ldots+\beta_s}}{M^{(1/j)}_{(\beta_1,\ldots,\beta_s,0,\ldots0)}}\\
\nonumber
&&+\sup_{\beta\in\N_0^d}\log
\frac{(Ah)^{\beta_{s+1}+\ldots+\beta_d}\alpha_{s+1}^{\beta_{s+1}/2}\cdots
\alpha_d^{\beta_d/2}}{M^{(1/j)}_{(0,\ldots,0,\beta_{s+1},\ldots,\beta_d)}}\\
\label{cjh}
\leq&&\omega_{\bar{\mathbf{M}}^{(1/j)}}(Ah)+\omega_{\mathbf{M}^{(1/j)}}(j\bar{\alpha}^{1/2})
\eeqs
for $\{\bar{M}^{(1/j)}_\beta\}_{\beta\in\N_0^s}$ defined by $\bar{M}^{(1/j)}_\beta:=
M^{(1/j)}_{(\beta_1,\ldots,\beta_s,0,\ldots,0)}$, where we have chosen $j\geq Ah$, taking into account that the associate function is increasing on $(0,+\infty)^d$. Note also that if $j\geq j_0$ then $\mathbf{M}^{(1/j)}\leq
\mathbf{M}^{(1/{j_0})}$ and hence $\omega_{\mathbf{M}^{(1/j)}}\geq 
\omega_{\mathbf{M}^{(1/{j_0})}}$.
Since $\omega_{\bar{\mathbf{M}}^{(1/j)}}(Ah)$ is a new constant depending on $\ell$
($A$ and $j$ depend on $h$ that depends on $\ell$), substituting \eqref{cjh} into
\eqref{Q8} we have that
\beqs
\label{620bis}
\begin{split}
&\forall \ell\in\N\,\exists j\in\N, C\geq1\,\mbox{s.t.}\,\forall\bfc\in\Lambda_{(\calM)}\\
&\sup_{\alpha\in\N_0^d}|c_\alpha|e^{\omega_{\bfN^{(1/\ell)}}(\ell\alpha^{1/2})}\leq C
\sup_{\alpha\in\N_0^d}|c_\alpha|e^{\omega_{\bfM^{(1/j)}}(j\bar{\alpha}^{1/2})}.
\end{split}
\eeqs

For $\beta\in\N_0^d$ we now consider $\bfc^\beta=(c_\alpha^\beta)_{\alpha\in\N_0^d}$
defined by
\beqsn
c_\alpha^\beta:=\delta_{\alpha\beta}=\begin{cases}
1 & \mbox{if}\ \alpha=\beta\cr
0 & \mbox{if}\ \alpha\neq\beta.
\end{cases}
\eeqsn

Clearly $\bfc^\beta\in\Lambda_{(\calM)}$ since
\beqsn
\sup_{\alpha\in\N_0^d}|c_\alpha^\beta|e^{\omega_{\bfM^{(1/h)}}(h\alpha^{1/2})}=
e^{\omega_{\bfM^{(1/h)}}(h\beta^{1/2})}<+\infty
\eeqsn
for every $\beta\in\N_0^d$ and $h\in\N$, by assumption \eqref{L12B}
which implies that $\bfM^{(1/h)}$ satisfies \eqref{limMalpha} and hence
$\omega_{\bfM^{(1/h)}}(t)<+\infty$ for all $t\in\R^d$.

We apply \eqref{620bis} to the sequences $\bfc^\beta$ and get that for all
$\ell\in\N$ there exist $j\in\N$ and $C\geq 1$ such that for all $\beta\in\N_0^d$ we have
\beqsn
e^{\omega_{\bfN^{(1/\ell)}}(\ell\beta^{1/2})}\leq C
e^{\omega_{\bfM^{(1/j)}}(j\bar\beta^{1/2})}
\eeqsn
or equivalently
\beqs
\label{Q9}
\omega_{\bfN^{(1/\ell)}}(\ell\beta^{1/2})
\leq\log C+\omega_{\bfM^{(1/j)}}(j\bar\beta^{1/2}).
\eeqs

If $t\in(0,+\infty)^d$ then there exist $\beta\in\N_0^d$ with $\beta^{1/2}<t\leq(\beta+\mathbf1)^{1/2}$ so that, from \eqref{Q9}: 
\beqsn
\omega_{\bfN^{(1/\ell)}}(\ell t)\leq&&
\omega_{\bfN^{(1/\ell)}}(\ell(\beta+\mathbf1)^{1/2})
\leq \log C+\omega_{\bfM^{(1/j)}}(j(\overline{\beta+\mathbf1})^{1/2})\\
\leq&&  \log C+\omega_{\bfM^{(1/j)}}(2j\beta^{1/2})
\leq  \log C+\omega_{\bfM^{(1/j)}}(2jt)
\eeqsn
since the associate function is increasing on $(0,+\infty)^d$ and $(\overline{\beta+\mathbf1})^{1/2}
\leq2\beta^{1/2}$.

It finally follows from \eqref{55bis} and \eqref{Q3} that for all $\ell\in\N$ there exist $j\in\N$
and $C\geq1$ such that
\beqsn
N_\alpha^{(1/\ell)}\geq&&\sup_{t\in(0,+\infty)^d}\frac{t^\alpha}{\exp
\omega_{\bfN^{(1/\ell)}}(t)}
=\sup_{s\in(0,+\infty)^d}\frac{(s\ell)^\alpha}{\exp
\omega_{\bfN^{(1/\ell)}}(s\ell)}\\
\geq&&\frac{\ell^{|\alpha|}}{C}\sup_{s\in(0,+\infty)^d}\frac{s^\alpha}{\exp
\omega_{\bfM^{(1/j)}}(2js)}
=\frac{\ell^{|\alpha|}}{C}\sup_{t\in(0,+\infty)^d}\frac{\left(\frac{t}{2j}\right)^\alpha}{\exp
\omega_{\bfM^{(1/j)}}(t)}\\
=&&\frac 1C\left(\frac{\ell}{2j}\right)^{|\alpha|}M_\alpha^{(1/j)}
\eeqsn
and we have thus proved that $\calM(\preceq)\calN$, since the sequences $\bfM^{(1/j)}$ and $\bfN^{(1/\ell)}$ are normalized.
\end{proof}

\begin{Rem}
\begin{em}
\label{rem62}
In Theorem~\ref{th42} we used Theorem~\ref{cor1} to characterise
the inclusion relations of the spaces for any dimension $d$, which was not possible in the 
analogous results \cite[Theorems~4.4, 4.5, 4.6]{BJOS-fuzzy},  where we needed $d=1$ in one 
implication since 
formula \eqref{Q3} was not available for the general anisotropic case.
\end{em}
\end{Rem}

\vspace*{10mm}
{\bf Acknowledgments.}
The authors are sincerely grateful to Prof. Andreas Debrouwere for pointing out 
 that the argument in the first version of the proof of Theorem~\ref{th1} was incomplete.
Boiti and Oliaro were partially supported by the INdAM - GNAMPA Project 2023
``Analisi di Fourier e Analisi Tempo-Frequenza
di Spazi Fun\-zio\-na\-li ed Ope\-ra\-to\-ri", CUP\_\,E53C22001930001.
Boiti was partially supported by the Projects FAR 2020, FAR 2021, FAR 2022,
FIRD 2022 and FAR 2023 (University of Ferrara). Jornet is partially supported by the project PID2020-\-119457GB-\-100 funded by MCIN/AEI/10.13039/501100011033 and by ``ERDF A way of making Europe''. Schindl is supported 
by FWF (Austrian Science fund) project 10.55776/P33417.

\vspace{5mm}
\samepage{
{\bf Statements and Declarations:}
\begin{itemize}
\item
 {\em Conflict of interest.} The authors declare that
there is no conflict of interest.
\item{\em Data availability statements.} The manuscript has no associated data.
\end{itemize}}

\end{document}